\documentclass[a4paper,reqno]{amsart}
\usepackage{amsmath,amsfonts,amssymb,amsthm}

\usepackage{graphicx}
\usepackage{epstopdf}
\usepackage{cases}
\usepackage{mathrsfs}
\usepackage{parskip}
\usepackage{pst-all}
\usepackage{dsfont}
\usepackage{mathtools}
\newtheorem{theorem}{Theorem}
\newtheorem{lemma}[theorem]{Lemma}
\newtheorem{proposition}[theorem]{Proposition}

\newcommand{\indep}{{\perp\!\!\!\perp}}
\newcommand{\depend}{{\perp\!\!\!\perp\!\!\!\!\!\! /\,\,}}
\date{}
\author{Xiangying Chen}
\address{Institut f\"{u}r Algebra und Geometrie, Otto von Guericke Universit\"at Magdeburg, Magdeburg, Germany}
\email{xiangying.chen@ovgu.de}
\usepackage[backend=bibtex,style=alphabetic,backref=false,doi=false,maxbibnames=8,maxcitenames=3,minnames=1,
isbn=false,url=false]{biblatex}
\bibliography{matroidcibib.bib}
\subjclass[2020]{Primary 05B35; Secondary 52C40, 62B10, 62R01}

\title[An axiomatization of matroids as conditional independence models]{An axiomatization of matroids and oriented matroids as conditional independence models}
\begin{document}
	\begin{abstract}
		Matroids and semigraphoids are discrete structures abstracting and generalizing linear independence among vectors and conditional independence among random variables, respectively. Despite the different nature of conditional independence from linear independence, deep connections between these two areas are found and still undergoing active research. In this paper, we give a characterization of the embedding of matroids into conditional independence structures and its oriented counterpart, which lead to new axiom systems of matroids and oriented matroids.
	\end{abstract}
	\maketitle
	\section{Introduction}
	For a finite ground set $E$ we denote by \[\mathcal{A}_E:=\{ (ij|K):K\subseteq E,i\neq j\in E\backslash K \}\] the set of \emph{conditional independent statements} (\emph{CI-statements}) $(ij|K)$ on $E$. We usually use the ground set $[n]=\{1,\ldots,n\}$ and write $\mathcal{A}_n:=\mathcal{A}_{[n]}$ for convenience. A \emph{conditional independence structure} (\emph{CI-structure}) $\mathcal{G}\subseteq \mathcal{A}_E$ on $E$ is a subset of $\mathcal{A}_E$. 
	
		Notational convention: In this paper we use the ``Mat\'{u}\v{s} notation". Subsets (possibly empty) are denoted by upper case letters, singletons are denoted by lower case letters, concatenation of letters means disjoint union of sets. In particular, $i$ and $j$ are exchangeable in the notation $(ij|K)$ of a CI-statement. If the CI-structure $\mathcal{G}$ is clear in the context, we write $i\indep j|K$ iff $(ij|K)\in \mathcal{G}$ and $i\depend j|K$ iff $(ij|K)\notin \mathcal{G}$. If a condition is appeared as an axiom, it should be valid for all subsets and singletons of the ground set such that the expression makes sense, for example, subsets are assumed to be disjoint whenever they are written in concatenation.
	
	Let $\xi=(\xi_1,\ldots,\xi_n)$ be an $n$-dimensional random vector, and we denote ``random variables $\xi_i$ and $\xi_j$ are conditionally independent given $\{\xi_l:l\in K\}$" by $i\indep j|K$. Then the CI-structure $[[\xi]]:=\{(ij|K)\in\mathcal{A}_n: i\indep j|K\}$ satisfies the \emph{semigraphoid axiom} \cite{dawid1979conditional}
	\begin{enumerate}
		\item[(SG)] $i\indep j|K \wedge i\indep\ell |jK \Rightarrow i\indep\ell |K \wedge i\indep j|\ell K $.
	\end{enumerate}
	A CI-structure is a \emph{semigraphoid} if it satisfies the axiom (SG). As in the case of matroids and linear independence, there exist semigraphoids which does not come from the conditional independence among any random vector. The CI-structures coming from random vectors are not finitely axiomatizable \cite{studeny1992conditional}, however they are approximated by semigraphoids \cite{studeny1994semigraphoids}.
	
	%rank function of matroids, polymatroids, entropy
	%semigraphoids, ci of random variables
	%geometry of semigraphoids, fan
	In \cite{morton2009convex} a geometric characterization of semigraphoids is given and is referred to as \emph{convex rank tests}.  A CI-statement $(ij|K)\in\mathcal{A}_n$ is associated to the set of walls
	\begin{equation}\label{walls}
		\{\mathbf{x}\in\mathbb{R}^n:x_{k_1}\geq\cdots\geq  x_{k_s}\geq x_i=x_j\geq x_{\ell_1}\geq \cdots \geq x_{\ell_{n-s-2}}\}
	\end{equation}
 for $k_1\cdots k_s=K$ and $\ell_1 \cdots \ell_{n-s-2}=[n]\backslash ijK$, in the permutohedral fan $\Sigma_{A_n}$. Semigraphoids are exactly the subsets $\mathcal{G}$ of $\mathcal{A}_n$ such that the removal of all walls corresponding to the elements of $\mathcal{G}$ from the permutohedral fan $\Sigma_{A_n}$ results in a coarser fan. If the resulting fan is polytopal, that is, the normal fan of a polytope, then the semigraphoid is \emph{structural} or a \emph{semimatroid}, and the polytope is a \emph{generalized permutohedron} \cite{postnikov2009permutohedra}. Generalized permutohedra can be characterized in several ways \cite[Theorem~15.3]{postnikov2008faces}. A set function $h\colon 2^{[n]}\rightarrow \mathbb{R}$ is \emph{submodular} if $h(\emptyset)=0$ and
	\[
	h(A)+h(B)\geq h(A\cap B)+h(A\cup B)
	\]
	for any $A,B\subseteq [n]$. Submodular functions are bijectively corresponding to the generalized permutohedra as their support functions. Therefore a CI-structure $\mathcal{G}\subseteq \mathcal{A}_n$ is a semimatroid iff there is a submodular function $h\colon 2^{[n]}\rightarrow \mathbb{R}$ such that 
	\[
	\mathcal{G}=\{ (ij|K)\in\mathcal{A}_n: h(iK)+h(jK)=h(ijK)+h(K) \}=:[[h]].
	\]
	Semimatroids encode the combinatorial data of generalized permutohedra: Two generalized permutohedra define the same semimatroid iff their face lattices are same. The lattice of semimatroids ordered by inclusion is isomorphic to the dual of the face lattice of the cone of submodular functions \cite{studeny2016core}.
	
	Submodular functions appear and are well investigated in various branches of mathematics. They are important in the modeling of conditional independence as well \cite[Chap.~5]{studeny2006probabilistic}. If $\xi$ is an $n$-dimensional discrete random vector, that is, a random vector which takes finite number of values, and $h_\xi\colon 2^{[n]}\rightarrow \mathbb{R}$ maps every subset $S\subseteq [n]$ to the Shannon entropy of the subvector $\xi_S$, then $h$ is submodular and $[[\xi]]=[[h_\xi]]$ \cite{fujishige1978polymatroidal}. The same also holds for the measures with finite multiinformation \cite[Corollary~2.2]{studeny2006probabilistic}, which contain the class of discrete random vectors and the class of regular normally distributed random vectors. 
	
	A submodular function $r\colon 2^{[n]}\rightarrow \mathbb{Z}$ is the \emph{rank function} of a \emph{matroid} $M$ on $[n]$ if it is integer-valued, subcardinal $r(A)\leq |A|$ $\forall A\subseteq [n]$ and monotonic $r(A)\leq r(B)$ $\forall A\subseteq B\subseteq [n]$. A set $S \subseteq [n]$ is \emph{independent} in $M$ if $r(S)=|S|$ and otherwise \emph{dependent}. Minimal dependent sets are \emph{circuits} and maximal independent sets are \emph{bases} of $M$. The basis polytope $\mathcal{P}_M\subseteq [0,1]^n$ of the matroid $M$ is the convex hull of the indicator vectors $\mathbf{e}_B=\sum_{i\in B}\mathbf{e}_i\in\mathbb{R}^n$ of all bases $B$ of $M$. It is the generalized permutohedron whose support function is $r$. 
	
	One of the most fascinating facts of matroids is that they can be defined in hundreds of different ways. The axiom systems arise from various areas of mathematics, they are elegant and define different natural looking objects that are indeed equivalent, however the equivalence can be far from obvious. This equivalence of axiomatically defined objects is called the ``cryptomorphism". For instance, rank functions of matroids are cryptomorphic to polytopes whose vertex coordinates are 0 or 1 and edges are parallel to $\mathbf{e}_i-\mathbf{e}_j$, by taking their basis polytopes. Classic axiom systems of matroids and the cryptomorphic among them can be found in the Appendix in \cite{white1986theory}. Recently, matroids are proven to be cryptomorphic to Stanley-Reisner rings whose symbolic powers are Cohen-Macaulay \cite{varbaro2011symbolic,minh2011cohen} (see \cite{terai2012cohen} for more commutative-algebraic axiom systems), tropical varieties of degree 1 \cite{fink2013tropical}, supports of multiaffine Lorentzian polynomials \cite{branden2020lorentzian} and simplicial complexes whose combinatorial atlases are hyperbolic \cite{chan2022introduction}. In this paper, we show that matroids are cryptomorphic to semigraphoids satisfying a single additional axiom (MCI) and oriented matroids are cryptomorphic to oriented CI-structures satisfying the axioms (OCI1)-(OCI5). We assume that the reader is familiar with the terminologies and basic facts about matroids and oriented matroids, which can be found in \cite{oxley2006matroid,white1986theory,bjorner1999oriented}.

	%semimatroids, generalized permutohedra, matroid polytope
	%cryptomorphy
	\section{Matroids as CI-structures}

	A matroid $M$ on $[n]$ with rank function $r$ defines a CI-structure 
	$[[M]]:=[[r]]=\{ (ij|K)\in\mathcal{A}_n:r(iK)+r(jK)=r(ijK)+r(K) \}$. In other words, $[[M]]$ is the semimatroid of the submodular function $r$. 
	
	If $M$ is loopless, one can recover its rank function from its CI-structure $[[M]]$ recursively by $r(\emptyset)=0$, $r(\textrm{singleton})=1$ and
	\begin{equation}\label{recursion}
		r(ijK)=\begin{cases}
			r(iK)+r(jK)-r(K) & i \indep j |K\\
			r(iK)+r(jK)-r(K)-1 & i\depend j|K.
		\end{cases}
	\end{equation}
Alternatively, we can recover the independent sets $\mathcal{I}(M)$ of the matroid $M$ by 
\begin{equation}\label{indsetsmatroidfromci}
	\mathcal{I}(M)=\{S\subseteq [n]:\mathcal{A}_S\subseteq [[M]]\},
\end{equation}

where $\mathcal{A}_S:=\{ (ij|K):K\subseteq S,i\neq j\in S\backslash K \}$.

Replacing a loop $i\in [n]$ by a coloop and replacing a coloop $i$ by a loop correspond to the translation of matroid polytope by $-\mathbf{e}_i$ and $\mathbf{e}_i$, respectively. They do not affect the normal fan and thus the semimatroid stays unchanged. In this paper we consider loopless matroids as the representatives in every class of matroids on $[n]$ whose base polytopes have the same normal fan.

In \cite{matuvs1993probabilistic}, Mat\'{u}\v{s} connected conditional independence to matroid theory by embedding matroids into CI-structures and investigating their probabilistic representability, and introducing the matroid theoretic tools to the research of conditional independence. He stated that a CI-structure can be defined by a matroid iff the sets obtained by (\ref{indsetsmatroidfromci}) are the independent sets of a matroid whose CI-structure is again the original one. Here we characterize these CI-structures by axioms in type of inference rules.
	\begin{theorem}\label{matroidci}
		A CI-structure $\mathcal{G}\subseteq \mathcal{A}_n$ is defined by a loopless matroid $M$ iff it satisfies
		\begin{enumerate}
			\item[\normalfont{(MCI)}] $i\depend j|K\Rightarrow i\indep\ell |jKL$,
			\item[\normalfont{(SG)}] $i\indep j|K \wedge i\indep\ell |jK \Rightarrow i\indep\ell |K \wedge i\indep j|\ell K $.
		%	\item[\textrm{(CI3)}] $i\indep j|K\Rightarrow i\indep \ell|K \vee j\indep \ell |K$.
		\end{enumerate}
		Moreover, the correspondence between matroids and CI-structures satisfying \normalfont{(MCI)} and \normalfont{(SG)} is one-to-one.
			\end{theorem}
		\begin{proof}
			Let $M$ be a matroid on $[n]$ with rank function $r$. The rank function axioms of matroids imply that $i\depend j|K$ in $[[M]]$ iff 
			\[
			r(iK)=r(jK)=r(ijK)=r(K)+1.
			\]
			From the second equality and submodularity $r(ijK)+r(jKL)\geq r(ijKL)+r(jK)$ and monotonicity $r(ijKL)\geq r(jkL)$ of $r$, we have $r(ijKL)=r(jKL)$. This contradicts $i\depend \ell |jKL$. So is (MCI) proven. The condition (SG) is the semigraphoid axiom. It is always satisfied by semimatroids.
			
		%	Now assume that $i\depend \ell|K$ and $j\depend \ell |K$. Then we have 
		%	\[r(iK)=r(jK)=r(\ell K)=r(i\ell K)=r(j\ell K)=r(K)+1.\]
		%	By submodularity, 
		%	\[
		%	r(ij\ell K)\leq r(i\ell K)+r(j\ell K)-r(\ell K)=r(\ell K).
		%	\]
		%	So we have $r(ij\ell K)=r(\ell K)=r(iK)=r(ijK)$ by monotonicity. Thus we got all required equalities for $i\depend j|K$.
			
			Now Let $\mathcal{G}\subseteq \mathcal{A}_n$ be a CI-structure satisfying (MCI) and (SG). We need to show that a function $r\colon 2^{[n]}\rightarrow \mathbb{N}$ is uniquely defined by the recursion (\ref{recursion}) and it is the rank function of a matroid. 
			
			To show the well-definedness of $r(S)$ we apply induction on the cardinality $c$ of $S$. For $c\in\{0,1\}$ the function value $r(S)$ is given by the initial condition. For $c=2$, $r(ij)=1$ if $i\depend j|$ and 2 otherwise. Assume that $c\geq 3$ and $r(S)$ is uniquely defined for any $S$ with $|S|<c$. What is left to check is that we got the same value $r(ijkL)$ for $|ijkL|=c$ by applying (\ref{recursion}) to different conditional (in)dependence statements, namely,
			\begin{align*}
				&\begin{cases}
					r(ikL)+r(jkL)-r(kL) & i \indep j |kL\\
					r(ikL)+r(jkL)-r(kL)-1 & i\depend j|kL
				\end{cases}\\
			=&\begin{cases}
				r(ijL)+r(jkL)-r(jL) & i \indep k |jL\\
				r(ijL)+r(jkL)-r(jL)-1 & i \depend k |jL.
			\end{cases}
			\end{align*}
		\begin{enumerate}
			\item $i \indep j |kL$ and $i \depend k |jL$: From (SG) we have $i\depend k|L$. By $i \depend k |jL$ and (MCI) we have $i\indep j|L$. Therefore
			\begin{align*}
				r(ikL)+r(jkL)-r(kL)&=r(iL)-r(L)-1+r(jkL)\\&=r(ijL)+r(jkL)-r(jL)-1.
			\end{align*}
		The case $i \indep k |jL$ and $i \depend j |kL$ follows by symmetry.
			\item  $i \depend j |kL$ and $i \depend k |jL$: By (MCI) we have $i\indep j|L$ and $i\indep k|L$.
			\item  $i \indep j |kL$ and $i \indep k |jL$: If $i\indep j|L$, by (SG) we have $i\indep k|L$. If $i\depend j|L$, we have $i\depend k|L$ instead by (SG). In both cases the equality follows.
		\end{enumerate}
			
			To show that $r$ is the rank function of a matroid, what is left to show is $r(i_1\cdots i_s)-r(i_2\cdots i_s)\in \{0,1\}$. Applying (\ref{recursion}) $s-2$ times, we have
			\begin{align*}
				r(i_1\cdots i_s)-r(i_2\cdots i_s)=r(i_1)-r(\emptyset)-a=1-a,
			\end{align*}
		where 
		\[
		a=\left|\{ (i_1i_s|i_2\cdots i_{s-1}),(i_1i_{s-1}|i_2\cdots i_{s-2}),\ldots, (i_1i_2|) \}\backslash \mathcal{G}\right|.
		\]
		The condition (MCI) implies that at most one of \[
		(i_1i_s|i_2\cdots i_{s-1}), (i_1i_{s-1}|i_2\cdots i_{s-2}), \ldots,(i_1i_3|i_2), (i_1i_2|)
		\] can be a conditional dependent statement, i.e. not in $\mathcal{G}$, therefore $a$ is either 0 or 1. 		 	
		\end{proof}

We remark that sometimes it may be more convenient to write the condition (MCI) in the following form \[
i\indep j|K \vee i\indep \ell| jKL \quad \forall \, i\ell jKL\subseteq [n].
\] Since semigraphoids satisfying (MCI) are cryptomorphic to loopless matroids, we call such CI-structures \emph{matroids} as well, and we can read off the CI-structure from other formulations of the matroid. The following lemma will be used repeatedly. 
	\begin{lemma}\label{cdcircle}
		For any $K\subseteq [n]$ and $i\neq j\in [n]\backslash K$, $i\depend j|K$ iff there is a circuit $C$ such that $ij\subseteq C\subseteq ijK$ and every circuit in $ijK$ contains either both $i,j$ or none of them.
	\end{lemma}
	\begin{proof}
		Let $C$ be a circuit such that $i\in C\subseteq iK$. Then $r(C)=r(C\backslash i)$. Together with $r(iK)=r(K)+1$ it contradicts the submodularity
		\[
		r(C)+r(K)\geq r(iK)+r(C\backslash i).
		\]
		If there is no circuit in $ijK$ containing $i,j$, then $i$ and $j$ are coloops in $ijK$. Then $r(ijK)=r(iK)+1$, so we have $i\indep j|K$.
		
		Now assume that $i\indep j|K$ and there is a circuit $ijK'$ such that $K'\subseteq K$. We want to show that there exists a circuit containing exactly one of $i,j$. From $i\indep j|K$ we have 
		\[
		r(iK)+r(jK)=r(ijK)+r(K),
		\]
		and as $ijK'$ is a circuit, $r(ijK')=r(jK')$. By submodularity
		\[
		r(jK)+r(ijK')\geq r(ijK)+r(jK')
		\]
		and monotonicity, we have $r(jK)=r(ijK)$ and therefore $r(iK)=r(K)$. So there is a circuit $C'$ such that $i\in C'\subseteq iK$.
	\end{proof}

We summarize the following characterizations for the conditional dependence $(ij|K)\notin [[M]]$. They are merely straightforward reformulations of the statements stated previously.

\begin{proposition}\label{dependenceequiv}
	Let $M$ be a matroid on $[n]$ and $[[M]]\subseteq \mathcal{A}_n$ be the CI-structure associated to $M$. The following are equivalent for any $K\subseteq [n]$ and $i\neq j\in[n]\backslash K$:
	\begin{enumerate}
		\item $i\depend j|K$.
		\item $r(iK)=r(jK)=r(ijK)=r(K)+1$.
		\item There exists a circuit $C$ of $M$ such that $ij\subseteq C\subseteq ijK$ and every circuit in $ijK$ contains either both $i,j$ or none of them. 
		\item The set $ij$ is a cocircuit (or equivalently, $K$ is a hyperplane) in the restriction of $M$ to $ijK$.
		\item For any basis $B$ of $K$, $iB$ and $jB$ are bases of $ijK$.
		%	\item there exist subsets $B\subseteq K$ and $A\subseteq [n]\backslash ijK$ such that $iAB$ and $jAB$ are bases of $M$,
		\item There exists a basis $B$ of $K$ such that $iB$ and $jB$ are bases of $ijK$.
		\item Any wall of $\Sigma_{A_n}$ corresponding to $(ij|K)$ by (\ref{walls}) is in some wall of the normal fan $\mathcal{N}(\mathcal{P}_M)$ of the base polytope $\mathcal{P}_M$ of $M$.
	\end{enumerate}
\end{proposition}

In \cite{matuvs1993probabilistic}, the notions of matroid operations are adapted to conditional independence. Let $\mathcal{G}\subseteq \mathcal{A}_n$ be a CI-structure. The \emph{deletion} and \emph{contraction} of $\mathcal{G}$ by $A\subseteq [n]$ are
\begin{align*}
	&\mathcal{G}\backslash A:= \{ (ij|K)\in \mathcal{G}: ijK\subseteq [n]\backslash A \},\\
	&\mathcal{G}/ A:= \{ (ij|K)\in \mathcal{A}_{[n]\backslash A}: (ij|KA)\in\mathcal{G} \},
\end{align*}
respectively. They reflect the marginalization and conditioning in probability theory. The \emph{dual} of $\mathcal{G}$ is $\mathcal{G}^\ast:=\{ (ij|[n]\backslash ijK):(ij|K)\in\mathcal{G} \}$. The usual rules for matroid operations also work for CI-structure operations, e.g. deletion and contraction commute and are dual operations of each other. A CI-structure $\mathcal{G'}$ is a \emph{minor} of a CI-structure $\mathcal{G}$ if $\mathcal{G}'$ can be obtained from $\mathcal{G}$ by applying any sequence of restriction and contraction. The \emph{direct sum} of two CI-structures $\mathcal{G}_1\subseteq \mathcal{A}_{E_1}$ and $\mathcal{G}_2\subseteq \mathcal{A}_{E_2}$ is the CI-structure
\begin{align*}
\mathcal{G}_1\oplus\mathcal{G}_2=\{ (ij|K)\in \mathcal{A}_{E_1E_2}: i\in E_1,j\in E_2 &\textrm{ or } ij\subseteq E_1,(ij|K\cap E_1)\in \mathcal{G}_1\\&\textrm{ or } ij\subseteq E_2,(ij|K\cap E_2)\in \mathcal{G}_2\}
\end{align*}
on $E_1E_2$.

The operations on CI-structures are compatible with the operations on matroids, namely, $[[M\backslash A]]=[[M]]\backslash A$, $[[M/ A]]=[[M]]/ A$ and $[[M^\ast]]=[[M]]^\ast$ for any matroid $M$ on $[n]$ and any $A\subseteq [n]$, and $[[M_1\oplus M_2]]=[[M_1]]\oplus [[M_2]]$ for any matroids $M_1,M_2$ on disjoint ground sets \cite{matuvs1993probabilistic}. In particular, being a matroid is a minor-closed property of CI-structures. In \cite{matuvs1997conditional}, classes of CI-structures are investigated via forbidden minors. The CI-structures that are matroids, although having a short axiomatization by (MCI) and (SG), cannot be characterized by a finite set of forbidden minors.

\begin{theorem}
	The class of CI-structures that are matroids cannot be characterized by a finite set of forbidden minors. 
\end{theorem}
\begin{proof}
	Let $\mathcal{G}_m:=\mathcal{A}_m\backslash \{ (12|),(ij|K):ijK=[m] \}$ for $m\geq 4$. In other words, $\mathcal{G}_m$ is the semimatroid defined by the sum of rank functions of matroids $U_{1,2}\oplus U_{m-2,m-2}$ and $U_{m-1,m}$, where $U_{r,m}$ is the rank-$r$ uniform matroid on $[m]$. Then $\mathcal{G}_m$ is not a matroid because (MCI) is violated by $(12|),(13|245\cdots m)\notin \mathcal{G}_m$. However, any proper minor of $\mathcal{G}_m$ is a matroid as $\mathcal{G}_m\backslash 1\cong \mathcal{G}_m\backslash 2\cong  [[U_{m-1}]]$, $\mathcal{G}_m\backslash i\cong [[U_{1,2}\oplus U_{m-3,m-3}]]$ for any $i\in [m]\backslash {1,2}$ and $\mathcal{G}/j\cong [[U_{m-2,m-1}]]$ for any $j\in[m]$.
\end{proof}

	\section{Oriented matroids as oriented CI-structures}
		An \emph{oriented CI-structure} on $[n]$ is a map $\sigma\colon\mathcal{A}_ n\rightarrow \{ -1,0,1 \}$. In \cite{boege2019geometry}, oriented gaussoids are introduced for modeling the signs of the partial correlations among regular normally distributed random vectors. General oriented CI-structures are introduced in \cite{boege2022gaussian}.
				
	 For notions regarding oriented matroids we use the notations in \cite{bjorner1999oriented}. A \emph{signed subset} of $[n]$ is a map $X\colon [n]\rightarrow \{-1,0,1\}$. Write $X^+:=X^{-1}(1)$, $X^-:=X^{-1}(-1)$ and $\underline{X}:=X^+X^-$. A collection $\mathcal{C}$ of signed subsets of $[n]$ is the set of \emph{signed circuits} of an \emph{oriented matroid} on $[n]$ if it satisfies
	\begin{enumerate}
		\item[\normalfont{(OC0)}] $\emptyset\notin \mathcal{C}$,
		\item[\normalfont{(OC1)}] $\mathcal{C}=-\mathcal{C}$,
		\item[\normalfont{(OC2)}] for all $X,Y\in\mathcal{C}$ with $\underline{X}\subseteq \underline{Y}$, either $X=Y$ or $X=-Y$,
		\item[\normalfont{(OC3)}] for all $X,Y\in \mathcal{C}$, $X\neq -Y$ and $e\in X^+\cap Y^-$ there is a $Z\in \mathcal{C}$ such that $Z^+\subseteq (X^+\cup Y^+)\backslash e$ and $Z^-\subseteq (X^-\cup Y^-)\backslash e$.
	\end{enumerate}
	If (OC0), (OC1) and (OC2) are satisfied, then (OC3) is equivalent to the following condition known as the \emph{strong elimination axiom}:
	\begin{enumerate}
		\item[\normalfont{(OC3')}] For all $X,Y\in \mathcal{C}$, $e\in X^+\cap Y^-$ and $f\in (X^+\backslash Y^-)\cup (X^-\backslash Y^+) $, there is a $Z\in \mathcal{C}$ such that $Z^+\subseteq (X^+\cup Y^+)\backslash e$, $Z^-\subseteq (X^-\cup Y^-)\backslash e$ and $f\in Z$.
	\end{enumerate}
	If $\mathcal{C}$ is a \emph{circuit signature} of a matroid $M$, that is, $\mathcal{C}$ consists of two opposite signed sets $X$ and $-X$ supported by $C$ for each circuit $C$ of $M$, then $\mathcal{C}$ clearly satisfies (OC0)-(OC2). In this case, we only need to check the $X,Y$ in (OC3) such that $\underline{X}$ and $\underline{Y}$ are a \emph{modular pair} in $M$, that is, $r(X)+r(Y)=r(X\cup Y)+r(X\cap Y)$. In this case, $Z$ is unique.
	
	Let $\mathcal{M}$ be an oriented matroid on $[n]$ and $\mathcal{C}$ be the set of signed circuits of $\mathcal{M}$. We associate an oriented CI-structure $\sigma_\mathcal{M}\colon\mathcal{A}_n\rightarrow\{-1,0,1\}$ to $\mathcal{M}$ by assigning $\sigma_\mathcal{M}(ij|K)=0$ whenever $(ij|K)\in [[\underline{\mathcal{M}}]]$, and otherwise $\sigma_\mathcal{M}(ij|K)=X(i)X(j)$ for any $X\in\mathcal{C}$ such that $ij\subseteq \underline{X}\subseteq ijK$. The following lemma ensures the well-definedness of $\sigma_\mathcal{M}$.
	
	\begin{lemma}\label{welldefomci}
		If $(ij|K)\notin [[\underline{\mathcal{M}}]]$, then $X(i)X(j)=Y(i)Y(j)$ for any $X,Y\in\mathcal{C}$ such that $ij\subseteq \underline{X},\underline{Y}\subseteq ijK$.
	\end{lemma}
\begin{proof}
	Suppose that $X(i)X(j)=-Y(i)Y(j)$. By (OC1), we can assume that $X(i)=Y(i)=1$ and $X(j)=Y(j)=-1$, that is, $j\in X^+\cap Y^-$ and $i\in X^+\backslash Y^-$. The strong elimination axiom (OC3') implies that there is a $Z\in \mathcal{C}$ such that $i\in \underline{Z}\subseteq iK$ and $j\notin \underline{Z}$, which contradicts Lemma~\ref{cdcircle}.
\end{proof}

We prove that oriented matroids can be axiomatized in terms of oriented conditional independence.
	\begin{theorem}
		An oriented CI-structure $\sigma\colon\mathcal{A}_ n\rightarrow \{ -1,0,1 \}$ is the associated oriented CI-structure $\sigma_\mathcal{M}$ of an oriented matroid $\mathcal{M}$ iff it satisfies
		\begin{enumerate}
			\item[\normalfont{(OCI1)}] $\sigma(ij|K)\neq 0\Rightarrow \sigma(i\ell |jKL)=0$,
			\item[\normalfont{(OCI2)}] $\sigma(ij|K)=\sigma(i\ell|jK)= 0\Rightarrow \sigma(ij|\ell K)=\sigma(i\ell|K)=0$,
			\item[\normalfont{(OCI3)}] $ \sigma(ij|K)\neq 0 \Rightarrow \sigma(ij|L)\in\{0,\sigma(ij|K)\}$ for any $L\subseteq K$ or $L\supseteq K$,
			\item[\normalfont{(OCI4)}] $\sigma(i\ell|K)\sigma(ij|K)\sigma(j\ell|K)\leq 0$,
			\item[\normalfont{(OCI5)}] $\sigma(i\ell|jK)\sigma(ij|\ell K)\sigma(j\ell|iK)\geq 0$.
		%	\item[(OCI3)] $\sigma(i\ell|jK)=0 \wedge \sigma(ij|\ell K)\neq 0\Rightarrow \sigma(ij|K)=\sigma(ij|\ell K)$,
		%	\item[(OCI4)] $\sigma(ij|K)\neq 0\Rightarrow \sigma(i\ell|K)=-\sigma(ij|K)\sigma(j\ell|K)$.
		%	\item[(OCI5)] $\sigma(ij|\ell K)\neq 0\Rightarrow \sigma(i\ell|jK)=\sigma(ij|\ell K)\sigma(j\ell|iK)$.
		\end{enumerate}
	In this case, the oriented matroid $\mathcal{M}$ can be determined uniquely from $\sigma$ by first recovering the underlying matroid $\underline{\mathcal{M}}$, and then assigning to each circuit $C$ of $\underline{\mathcal{M}}$ two opposite signed circuits
	\begin{equation}\label{circuitsignature}
		\pm (\{c_0,c\in C:\sigma(cc_0|C\backslash cc_0)=1\},\{c\in C:\sigma(cc_0|C\backslash cc_0)=-1\}),
	\end{equation}
	where $c_0\in C$ can be chosen arbitrarily in each $C$.
	\end{theorem}
	\begin{proof}
		Let $\mathcal{M}$ be an oriented matroid. By the construction of $\sigma_\mathcal{M}$, $\sigma^{-1}(0)$ is the CI-structure associated to the underlying matroid $\underline{\mathcal{M}}$, (OCI1) and (OCI2) are guaranteed by Theorem~\ref{matroidci}. The condition (OCI3) follows by definition and Lemma~\ref{welldefomci}. 
		
		If $\sigma(i\ell|K)\sigma(ij|K)\sigma(j\ell|K)\neq 0$, then by Lemma~\ref{cdcircle}, there are signed circuits $X,Y$ of $\mathcal{M}$ such that $ij\subseteq \underline{X}$, $\ell\notin \underline{X}$, $i\ell \subseteq  \underline{Y}$ and $j\notin \underline{Y}$. By (OC1) we can assume $i\in X^+\cap Y^-$. Then by (OC3) there is a signed circuit $Z$ of $\mathcal{M}$ such that $Z^+\subseteq (X^+\cup Y^+)\backslash e$ and $Z^-\subseteq (X^-\cup Y^-)\backslash e$. It follows from $\sigma(j\ell|K)\neq 0$ that $j\ell\subseteq  \underline{Z}$ and
		\[
		\sigma(j\ell|K)=Z(j)Z(\ell)=X(j)Y(\ell)=X(i)\sigma(ij|K)Y(i)\sigma(i\ell|K)=\sigma(ij|K)\sigma(i\ell|K),
		\]
		so (OCI4) is proven. 
		
		If $\sigma(i\ell|jK)\sigma(ij|\ell K)\sigma(j\ell|iK)\neq 0$, then by Lemma~\ref{cdcircle}, for any signed circuit $X$ of $\mathcal{M}$ such that $\underline{X}\subseteq ij\ell K$, either $ij\ell\subseteq \underline{X}$ or $\{i,j,l\}\cap\underline{X}=\emptyset$, and there exists a signed circuit $X$ such that $ij\ell \subseteq \underline{X}\subseteq ij\ell K$. Let $X$ be such a signed circuit. Then we conclude (CI5) as
		\[
		\sigma(i\ell|jK)\sigma(ij|\ell K)\sigma(j\ell|iK)=X(i)^2X(j)^2X(\ell)^2=1.
		\]
		
		Now let $\sigma\colon \mathcal{A}_n\rightarrow\{-1,0,1\}$ be such that (OCI1)-(OCI5) are satisfied, and let $M$ be the loopless matroid obtained from the CI-structure $\sigma^{-1}(0)\subseteq \mathcal{A}_n$ by Theorem~\ref{matroidci}. Let $\mathcal{C}$ be the circuit signature (\ref{circuitsignature}) of $M$. It follows by (OCI5) that the element $c_0$ in $(\ref{circuitsignature})$ can be chosen arbitrarily in each $C$.
		
		Let $X,Y\in \mathcal{C}$ be such that $\underline{X}$ and $\underline{Y}$ are a modular pair in $M$. Let $e\in X^+\cap Y^-$. By the construction (\ref{circuitsignature}) we have $X(x)=\sigma(xe|\underline{X}\backslash xe)$ for all $x\in \underline{X}\backslash e$ and $Y(y)=-\sigma(ye|\underline{Y}\backslash ye)$ for all $y\in \underline{Y}\backslash e$. 
		
		Let $C$ be the unique circuit of $M$ such that $C\subseteq (\underline{X}\cup \underline{Y})\backslash e$. Explicitly, $C=[n]\backslash\operatorname{cl}^*(([n]\backslash (\underline{X}\cup \underline{Y}))e)$, where $\operatorname{cl}^\ast$ is the closure operator of the dual matroid of $M$.
		
		 Let $f\in \underline{X}\backslash \underline{Y}$ and $g\in\underline{Y}\backslash\underline{X}$. The existence of $f$ and $g$ are ensured by the incomparability of circuits of a matroid. Because $C$ is the unique circuit of $M$ in $(\underline{X}\cup \underline{Y})\backslash e$, by the strong circuit elimination axiom of matroids, we have $f,g\in C$. Let $Z$ be the signed subset of $[n]$ supported by $C$ with $Z(f)=X(f)=\sigma(ef|\underline{X}\backslash ef)$ which is defined by (\ref{circuitsignature}), that is,
		 \[
		 	Z(z)=\sigma(zf|C\backslash zf)\sigma(ef|\underline{X}\backslash ef).
		 \] 
		 We show that $Z$ is the desired circuit in the condition (OC3). 
		 
		 Case 1: If $z\in C\cap \underline{X}\backslash \underline{Y}$, then $\sigma(zf|(\underline{X}\cup \underline{Y})\backslash zf)\neq 0$ because
		 \[
		 r(\underline{X}\cup \underline{Y})=r((\underline{X}\cup \underline{Y})\backslash z)=r((\underline{X}\cup \underline{Y})\backslash f)
		 \] 
		 and 
		 \begin{align*}
		 	r(\underline{X}\cup \underline{Y})-1\leq  r((\underline{X}\cup \underline{Y})\backslash zf)&\leq r(\underline{X}\backslash zf)+r(\underline{Y})-r(\underline{X}\cap \underline{Y})\\&=r(\underline{X})-1+r(\underline{Y})-r(\underline{X}\cap \underline{Y})=r(\underline{X}\cup \underline{Y})-1.
		 \end{align*}
		 It follows from $\sigma(zf|C\backslash zf),\sigma(zf|\underline{X}\backslash zf),\sigma(zf|(\underline{X}\cup \underline{Y})\backslash zf)\neq 0$ and (OCI3) and (OCI5) that  
		 \begin{align*}
		 	Z(z)&=\sigma(zf|C\backslash zf)\sigma(ef|\underline{X}\backslash ef)=\sigma(zf|(\underline{X}\cup \underline{Y})\backslash zf)\sigma(ef|\underline{X}\backslash ef)\\&=\sigma(zf|\underline{X}\backslash zf)\sigma(ef|\underline{X}\backslash ef)=\sigma(ez|\underline{X}\backslash ez)=X(z).
		 \end{align*}
	 
	 Case 2: If $z\in C\cap \underline{Y}\backslash \underline{X}$, we need to show $Z(g)=Y(g)$, and the rest of this case is same as Case 1. It follows from (OCI3) and (OCI5) that
	\begin{align*}
		Z(g)&=Z(f)\sigma(fg|C\backslash fg)=\sigma(ef|\underline{X}\backslash ef)\sigma(fg|C\backslash fg)\\&=\sigma(ef|\underline{X}\backslash ef)\sigma(fg|(\underline{X}\cup \underline{Y})\backslash efg)\\&=-\sigma(ef|\underline{X}\backslash ef)\sigma(ef|(\underline{X}\cup \underline{Y})\backslash efg)\sigma(eg|(\underline{X}\cup \underline{Y})\backslash efg)\\&=-\sigma(ef|\underline{X}\backslash ef)\sigma(ef|\underline{X}\backslash ef)\sigma(eg|\underline{Y}\backslash eg)=-\sigma(eg|\underline{Y}\backslash eg)=Y(g),
	\end{align*}
where all $\sigma$-values are nonzero because of Lemma~\ref{cdcircle}.

Case 3: If $z\in C\cap \underline{X}\cap \underline{Y}$, suppose in the contrary that $X(z)=Y(z)=-Z(z)$. We have
\begin{align*}
	\sigma(ez|\underline{X}\backslash ez)=X(z)=-Z(z)&=	-\sigma(zf|C\backslash zf)\sigma(ef|\underline{X}\backslash ef)\\&=-\sigma(zf|C\backslash zf)\sigma(zf|\underline{X}\backslash zf)\sigma(ez|\underline{X}\backslash ez),
\end{align*}
thus $\sigma(zf|C\backslash zf)\sigma(zf|\underline{X}\backslash zf)=-1$. Similarly, $\sigma(zg|C\backslash zf)\sigma(zg|\underline{Y}\backslash zg)=-1$ because
\begin{align*}
	\sigma(ez|\underline{Y}\backslash ez)&=-Y(z)=Z(z)=Z(g)\sigma(zg|C\backslash zg)=Y(g)\sigma(zg|C\backslash zg)\\&=-\sigma(eg|\underline{Y}\backslash eg)\sigma(zg|C\backslash zg)=-\sigma(ez|\underline{Y}\backslash ez)\sigma(zg|\underline{Y}\backslash zg)\sigma(zg|C\backslash zg).
\end{align*}
Therefore,
\begin{align*}
	\sigma(fg|C\backslash fg)&=\sigma(zg|C\backslash zg)\sigma(zf|C\backslash zf)=\sigma(zg|\underline{Y}\backslash zg)\sigma(zf|\underline{X}\backslash zf)\\&=\sigma(eg|\underline{Y}\backslash eg)\sigma(ez|\underline{Y}\backslash ez)\sigma(ef|\underline{X}\backslash ef)\sigma(ez|\underline{X}\backslash ez)\\&=\sigma(eg|\underline{Y}\backslash eg)(-Y(z))\sigma(ef|\underline{X}\backslash ef)X(z)\\&=-\sigma(eg|\underline{Y}\backslash eg)\sigma(ef|\underline{X}\backslash ef)=-Z(f)Z(g),
\end{align*}
	which contradicts (\ref{circuitsignature}).

We have shown that an oriented CI-structure $\sigma$ satisfying (OCI1)-(OCI5) defines an oriented matroid $\mathcal{M}$. The fact $\sigma=\sigma_\mathcal{M}$ follows from (OCI3) and Lemma~\ref{cdcircle}, namely, if $\sigma(ij|K)\neq 0$, then for any circuits $C,C'$ of $\underline{\mathcal{M}}$ with $ij\subseteq C,C'\subseteq ijK$, we have $\sigma(ij|C\backslash ij)=\sigma(ij|K)=\sigma(ij|C'\backslash ij)$.
	\end{proof}
Oriented matroids are known to be cryptomorphic to chirotopes, which abstract the possible signs of Pl\"ucker coordinates of points in the Grassmannian. By the following Proposition, $\sigma_{\mathcal{M}}$ can be obtained directly from the chirotope of $\mathcal{M}$.
\begin{proposition}
	Let $\chi\colon\binom{[n]}{r}\rightarrow \{-1,0,1\}$ be the chirotope of an oriented matroid $\mathcal{M}$ on $[n]$. The oriented CI-structure $\sigma_\mathcal{M}\colon \mathcal{A}_n\rightarrow \{-1,0,1\}$ associated to $\mathcal{M}$ can be obtained from the chirotope $\chi$ by
	\[
	\sigma_\mathcal{M}(ij|K)=-\chi(i,b_1,\ldots,b_s,a_1,\ldots,a_{r-s-1})\chi(j,b_1,\ldots,b_s,a_1,\ldots,a_{r-s-1}),
	\]
	where $b_1\cdots b_s$ is a basis of $K$ and $a_1\cdots a_{r-s-1}$ is a subset of $[n]\backslash ijK$ such that $\chi(i,b_1,\ldots,b_s,a_1,\ldots,a_{r-s-1})\neq 0$ whenever it exists, and $\sigma_\chi(ij|K)=0$ if no such $a_1\cdots a_{r-s-1}$ exists.
\end{proposition}
\begin{proof}
	Let $\sigma_\chi(ij|K)$ be the oriented CI-structure defined in the Proposition. It is easy to see that $\sigma_\chi(ij|K)=0$ iff there is a basis $b_1\cdots b_s$ of $K$ such that $iK$ or $jK$ is not a basis of $ijK$ in $\underline{\mathcal{M}}$. By Proposition~\ref{dependenceequiv}, this is equivalent to $(ij|K)\in[[\underline{\mathcal{M}}]]$. Now assume that $(ij|K)\notin [[\underline{\mathcal{M}}]]$. Let $C$ be the circuit in $\underline{\mathcal{M}}$ such that $ij\subseteq C\subseteq ijb_1\cdots b_s$, then $\sigma_\chi(ij|K)=\sigma_\chi(ij|C\backslash ij)$ coincides with the sign $\sigma(i,j)$ defined in \cite[Lemma~3.5.7~(i)]{bjorner1999oriented}, which gives back the original oriented matroid $\mathcal{M}$ by \cite[Lemma~3.5.7~(ii) and 3.5.10]{bjorner1999oriented}, therefore we have $\sigma_\chi=\sigma_{\mathcal{M}}$. The well-definedness of $\sigma_\chi$ follows as well.
\end{proof}
	\section{Remarks on representations of matroids}
	In \cite{matuvs1993probabilistic}, Mat\'u\v{s} introduced representations of matroids in probability theory and information theory which lead to new insight into representability of matroids. A matroid $M$ is \emph{probabilistically representable} if its CI-structure coincides with the conditional independence among a discrete random vector $\xi$, that is, $[[M]]=[[\xi]]=[[h_\xi]]$. As the rank function of any connected matroid spans an extreme ray in the cone of submodular functions, the probabilistic representability is equivalent to the entropicness in the case of connected matroids, where a matroid is \emph{entropic} if its rank function is a positive multiple of the entropy function of some discrete random vector. Linear and multilinear matroids are entropic \cite[Lemma~10]{matuvs1997conditional}, however the direct sum of Fano matroid and non-Fano matroid is algebraic but not entropic \cite{matuvs2018classes}. The rank function of any algebraic matroids is the pointwise limit of a sequence of entropy functions \cite{matuvsalgebraic}, such matroids are called \emph{almost entropic} \cite{matus2006two}. The dual of an almost entropic matroid is not necessarily almost entropic \cite{kaced2018information}, while a long standing open problem in matroid theory is whether the dual of any algebraic matroid is algebraic. 
	
	A Gaussian conditional independence, however, can never represent an interesting matroid. A semigraphoid $\mathcal{G}\subseteq \mathcal{A}_n$ is a \emph{gaussoid} if it satisfies the following axioms:
	\begin{enumerate}
			\item[\normalfont{(Int)}] $\{(ij|kL),(ik|jL)\}\subseteq \mathcal{G}\Rightarrow \{(ij|L),(ik|L)\}\subseteq \mathcal{G}$,
	\item[\normalfont{(Comp)}] $\{(ij|L),(ik|L)\}\subseteq \mathcal{G}\Rightarrow \{(ij|kL),(ik|jL)\}\subseteq \mathcal{G}$,
\item[\normalfont{(WT)}] $\{(ij|L),(ij|kL)\}\subseteq \mathcal{G}\Rightarrow (ik|L) \textrm{ or } (jk|L)\in\mathcal{G}$.
	\end{enumerate}
Gaussoids are introduced in \cite{lnvenivcka2007gaussian} as an abstraction of the conditional independence among regular normally distributed random variables: If $\xi$ is an $n$-dimensional regular normally distributed random vector whose covariance matrix is $\Sigma$, then $[[\xi]]=\{(ij|K)\in\mathcal{A}_n:\det\Sigma_{iK,jK}=0\}=:[[\Sigma]]$ is a gaussoid. The CI-structures coming from regular normally distributed random vectors are not finitely axiomatizable \cite{sullivant2009gaussian}, they are axiomatically approximated by gaussoids \cite{boege2022gaussoids}. In \cite[\S 7]{mohammadi2018generalized} it was computationally verified that for $3\leq n\leq 8$ no CI-structure corresponding to a connected matroid on $[n]$ is a gaussoid. We confirm this computation by showing that a loopless matroid is a gaussoid iff it is a direct sum of copies of uniform matroids $U_{1,1}$ and $U_{1,2}$.

	\begin{proposition}
		Let $M$ be a loopless matroid. Then $[[M]]$ is a gaussoid iff $M\cong U_{1,1}^{\oplus m_1}\oplus U_{1,2}^{\oplus m_2}$.
	\end{proposition}
\begin{proof}
	The necessity follows from the observation that $[[U_{1,1}]]=[[(1)]]$ and $[[U_{1,2}]]=[[\left(\begin{matrix}
		1 & 0.1\\
		0.1 & 1
	\end{matrix}\right)]]$ are gaussoids and $[[\Sigma_1]]\oplus[[\Sigma_2]]=[[\Sigma_1\oplus\Sigma_2]]$ for any positive definite matrices $\Sigma_1,\Sigma_2$ \cite[Lemma~3.14]{boege2022gaussian}.
	
	Now assume that $M$ is a loopless matroid on the ground set $E$ such that $[[M]]$ is a gaussoid. We deduce the following three inference rules for any $ij\ell K\subseteq E$ from (MCI), (SG), (Int) and (Comp):

Claim 1. $i\depend j|K \Rightarrow i\depend j|\ell K$. It follows by (MCI) from $i\depend j|K$ that $i\indep \ell|jK$ and $j\indep \ell|iK$. And by (Int) we get $i\indep \ell|K$. However by (SG) we have $i\depend \ell|K$ or $i\depend j|\ell K$ from $i\depend j|K$. Thus we have $i\depend j|\ell K$.

Claim 2. $i\depend j|\ell K \Rightarrow i\depend j| K$. It follows by (MCI) from $i\depend j|\ell K$ that $i\indep \ell|K$ and $j\indep \ell|K$. And by (Comp) we get $i\indep \ell|jK$. However by (SG) we have $i\depend \ell|jK$ or $i\depend j|K$ from $i\depend j|\ell K$. Thus we have $i\depend j|K$.

Claim 3. $i\depend j| \Rightarrow i\indep \ell |$. By applying (MCI) we get $i\indep \ell |j $ and $j\indep \ell |j $ from $i\depend j|$. Then $i\indep \ell|$ follows from (Int).

By Claim~3, the ground set $E$ can be partitioned into 2-element sets $\{a_\iota,b_\iota \}$, $\iota\in[m_2]$ and singletons $\{c_\iota\}$, $\iota \in[m_1]$ such that $a_\iota\depend b_\iota|$ for any $\iota\in [m_2]$ and $i\indep j|$ if $i,j\in E$ are in different blocks. By Claim~1 and Claim~2, for any two subsets $K,K'\subseteq E\backslash ij$ we have $i\indep j|K$ iff $i\indep j|K'$. Therefore, the loopless matroid $M$ is isomorphic to the direct sum of $m_2$ rank one uniform matroids on $\{a_\iota,b_\iota\}$, $\iota\in[m_2]$ and $m_1$ rank one uniform matroids on $\{c_\iota\}$, $\iota\in[m_1]$.
\end{proof}

\section*{Acknowledgments}
The author would like to thank Thomas Kahle for his comments. Funded by the Deutsche Forschungsgemeinschaft (DFG, German Research Foundation) - 314838170, GRK 2297 MathCoRe.

\printbibliography
\end{document}